\documentclass[11pt,a4paper,twoside]{amsart}

\usepackage{amsmath,amssymb,amsthm}
\usepackage{inputenc}
\usepackage{tikz} 
\usepackage{float}
\floatstyle{boxed} 
\usepackage{graphicx,subfigure}
\usepackage{caption}
\usepackage{pdfsync} 
\usepackage[T1]{fontenc} 
\usepackage{enumitem}
\usepackage{mathtools}
\usepackage{mathrsfs}
\usepackage{slashed}
\usepackage{xcolor}
\usepackage[linktocpage=true,colorlinks=true, linkcolor=blue, citecolor=blue, urlcolor=blue,hyperfootnotes=false]{hyperref}
\usepackage{cleveref}

\def\@begintheorem#1#2{\par\bgroup{\sc #1 \ #2. }  \it \\\ignorespace }
\def\@opargbegintheorem#1#2#3{\par\bgroup{\sc #1\ #2 \ (#3).}  \it  \ignorespace}
\def\@endtheorem{\egroup}

\theoremstyle{plain}
\newtheorem{theorem}{Theorem}[section]

\theoremstyle{definition}

\theoremstyle{remark}
\newtheorem{remark}[theorem]{Remark}

\theoremstyle{plain}

\theoremstyle{plain}
\newtheorem{lemma}[theorem]{Lemma}

\theoremstyle{plain}

\theoremstyle{plain}

\numberwithin{equation}{section}

\newcommand{\C}{{\mathbb C}}

\newcommand{\R}{{\mathbb R}}

\newcommand{\I}{\mathbb{I}}

\newcommand{\V}{\mathbf{V}}

\renewcommand{\S}{{\mathbf{S}}}

\newcommand{\Rt}{{\R}^3}
\newcommand{\SL}{\S\cdot L}

\newcommand{\hx}{{\hat{x}}}
\renewcommand{\H}{\mathcal{H}}

\DeclareMathOperator{\Realpart}{Re}
\renewcommand{\Re}{\Realpart}

{\left\lbrace\begin{array}{@{}l@{}}}%
{\end{array}\right.}

\DeclarePairedDelimiter{\abs}{\lvert}{\rvert}

\DeclarePairedDelimiter{\seq}{\lbrace}{\rbrace}

\newcommand{\Dirac}{\slashed{\mathcal{D}}}

\newcommand{\ignora}[1]{}
\newcommand{\verde}[1]{}

\title[A relativistic Hardy-type inequality with minimisers]
{A relativistic Hardy-type inequality with minimisers}
\date{\today}

  \author[L.~Fanelli]{Luca Fanelli}
\address[L.~Fanelli]{Ikerbasque, Basque Foundation for Science, 48011 Bilbao, Spain, \& Universidad del Pa\'is Vasco / Euskal Herriko Unibertsitatea, 48080 Bilbao, Spain, \& BCAM - Basque Center for Applied Mathematics, 48009 Bilbao, Spain}
\email{luca.fanelli@ehu.eus}

\author[F.~Pizzichillo]{Fabio Pizzichillo}
\address[F.~Pizzichillo]{Departamento de Matemática e Informática Aplicadas a la Ingeniería Civil y Naval,
Universidad Politécnica de Madrid\\
E.~T.~S.~I.~de Caminos, Canales y Puertos,
Calle del Profesor Aranguren 3,
Madrid, 28040 (Spain) \& BCAM - Basque Center for Applied Mathematics, 48009 Bilbao, Spain}
\email{fabio.pizzichillo@upm.es}


\usepackage{csquotes}

\begin{document}
\begin{abstract}
In this paper, we prove a sharp, weighted Hardy-type inequality for the Dirac operator. 
A key feature of our result is that the inequality is not only sharp but also attained, and we construct explicit minimizers that satisfy the equality case. This extends previous work on the spectral properties of Dirac operators, especially in the context of relativistic quantum mechanics and Coulomb-like potentials.
\end{abstract}
\maketitle

\section{Introduction and main results}
First formulated by Hardy in 1920 \cite{hardy1920note}, the Hardy inequality is a fundamental result in analysis, particularly relevant to the study of partial differential equations and functional spaces. It provides a bound for integrals involving a function and its gradient.
There are multiple versions of the Hardy inequality (see \cite{kufner2006prehistory}), but the most common modern form is expressed in $\R^d$ for dimensions $d \geq 3$, as follows: 
\begin{equation}\label{eq:classical}
\int_{\R^d}|\nabla \psi|^2\,dx\geq\frac{(d-2)^2}{4}\int_{\R^d}\frac{|\psi|^2}{|x|^2}\,dx,
\end{equation}
for every weakly differentiable function $\psi$ in $\R^d$ such that $|\nabla \psi|\in L^2(\R^d)$ and decaying to zero at infinity. 
It is well known that the inequality \eqref{eq:classical} is \emph{sharp}, meaning the dimensional constant is optimal, and no additional non-negative term can be added on the right-hand side. However, the inequality is \emph{not attained}, in the sense that it does not exist any function $\psi$ such that $\nabla\psi\in L^2(\R^d)$ and equality holds. This absence of minimizers has motivated research into improved versions of \eqref{eq:classical} and related inequalities. For instance, advancements have been made by replacing the domain $\R^d$ with any open bounded subset $\Omega$ containing $0$, and assuming that $\psi$ lies in the Sobolev space $H^1_0(\Omega)$ (see \cite{MR2457816}).

Attained inequalities are worth because they establish when the sharpest bounds are achieved, providing precise insights for analysis and optimization. These inequalities help to identify exact conditions under which equality holds, offering practical tools for solving mathematical problems.

In Quantum Mechanics, the Hardy inequality plays a crucial role in studying systems with singular potentials, such as the Coulomb potential, by providing bounds on the kinetic energy operator. 
This paper aims to extend these results by proving a weighted Hardy-type inequality with minimisers for the Dirac operator, which is an unusual feature compared to the classical case \eqref{eq:classical}.

For $m \geq 0$, the \emph{free Dirac operator} in $\Rt$ is defined by 
\[
\Dirac_m:=-i\alpha\cdot\nabla+m\beta =-i \sum_{j=1}^3 \alpha_j \partial_j + m\beta, 
\]
where
\begin{equation*}
  \beta:=
  \begin{pmatrix}
\mathbb{I}_2&0\\
0&-\mathbb{I}_2
\end{pmatrix},
\quad
\alpha:=(\alpha_1,\alpha_2,\alpha_3),
\quad
\alpha_j:=\begin{pmatrix}
0& {\sigma}_j\\
{\sigma}_j&0
\end{pmatrix}\quad \text{for}\ j=1,2,3,
\end{equation*}
and 
$\sigma_j$ are the \emph{Pauli matrices}
\[
\quad{\sigma}_1 =
\begin{pmatrix}
0 & 1\\
1 & 0
\end{pmatrix},\quad {\sigma}_2=
\begin{pmatrix}
0 & -i\\
i & 0
\end{pmatrix},
\quad{\sigma}_3=
\begin{pmatrix}
1 & 0\\
0 & -1
\end{pmatrix}.
\]
It is well known (see \cite{thaller}) that $\Dirac_m$ is self-adjoint on $H^1(\R^3;\C^4)$ and essentially
self-adjoint on $C_c^{\infty}(\R^3;\C^4)$. Moreover 
\[
\sigma(\Dirac_m)=\sigma_{ess}(\Dirac_m)=(-\infty,-m]\cup[m,+\infty).
\] 

One can apply \eqref{eq:classical} to explore the spectral properties of the Dirac operator perturbed by the Coulomb potential $\V_C$ defined as:
\begin{equation}\label{eq:coulomb.potential}
\V_{C}(x)=\frac{\nu}{|x|}\I_4,
\end{equation}
with $\nu=e^2Z/\hbar$, 
where 
$Z$ is the atomic number, $e$ is the charge of the electron and $\hbar$ is the Plank's constant (we set $\hbar=1$).
The operator $\Dirac_m + \V_C$ describes relativistic spin--$\frac12$ particles in the external electrostatic field of an atomic nucleus. 
Alongside the Kato-Rellich theorem, one finds that $\Dirac_m + \V_C$ is self-adjoint on $H^1(\Rt;\C^4)$ for $|\nu| < 1/2$, and essentially self-adjoint on $C^\infty_c(\Rt;\C^4)$ for $|\nu|\leq 1/2$, as discussed in \cite{kato1951fundamental}. However, this result is physically insufficient for describing all known atoms. Extending the range to $|\nu| \leq \sqrt{3}/2$ is necessary, as shown in \cite{rellich1953eigenwerttheorie, weidmann1971oszillationsmethoden, rejto1971some, gustafson1973some, schmincke1972essential}.
These results suggest that the Hardy inequality \eqref{eq:classical}  is not optimal to analyse the self-adjointness of perturbed Dirac operators, since it does not take into account the matrix structure of the Dirac operator. 
Hence, more suitable Hardy-type inequalities are needed. Kato addressed this in \cite{kato1983holomorphic}, proving a $4$-spinor Hardy inequality, first conjectured by Nenciu in \cite{nenciu1976self}:
  \begin{equation}\label{eq:hardy.ADV}
    \int_{\Rt}\frac{|\psi|^2}{|x|}\,dx \leq \int_{\Rt}\abs*{(-i\alpha\cdot\nabla+m\beta)\psi}^2|x|\,dx, \quad\text{for}
    \ \psi\in C^\infty_c(\Rt;\C^4).
  \end{equation}
This inequality enabled Kato to tackle the self-adjointness of the Dirac-Coulomb operator for general Hermitian matrix-valued potentials $\V$ satisfying
  \begin{equation}\label{eq:cond.V}
    \sup_{x\in\Rt} |x||\V(x)|=: \nu <1.
  \end{equation}
Further properties of the self-adjointness of the Dirac-Coulomb operator have since been explored in \cite{adv2013self,adv2018erratum} again using inequality  \eqref{eq:hardy.ADV}.

A refined version of \eqref{eq:hardy.ADV} was later derived in \cite{KlausSeveral}:
\begin{equation}\label{eq:hardy.con.a}
\frac{m^2-a^2}{m^2}\int_{\Rt}\frac{|\psi|^2}{|x|}\,dx
\leq
\int_{\Rt}\left|(-i\alpha\cdot\nabla+m\beta-a)\psi\right|^2|x|\,dx, \quad\text{for}\ a\in(-m,m)
\end{equation}
This inequality opened the way to prove the Birman-Schwinger principle for the Dirac-Coulomb operator. More recently, \eqref{eq:hardy.con.a} was re-proved using elementary integration by parts in \cite{coulombluis}.
This proof established that the inequality is sharp and is attained for $a \neq 0$. As a result, the authors were able to formulate conditions on the \emph{ground-state energy}, showing that it is achieved if and only if the potential $V$ satisfies specific rigidity conditions. In particular, for an electrostatic potential, these conditions imply that $V$ must be the Coulomb potential \eqref{eq:coulomb.potential}.
Nevertheless, all this results are only valid  in the \emph{sub-critical} case, that is when $\nu<1$ in \eqref{eq:cond.V}.

To extend beyond this restriction, we refer to the seminal work of Dolbeault, Esteban, and Séré \cite{dolbeault2000eigenvalues}. They developed a min-max formula for determining eigenvalues within the gap of the essential spectrum of the Dirac operator perturbed by Coulomb-like potentials $\V$ satisfying:
\begin{equation}
  \label{eq:des.potential}
  \V(x):= V(x) \I_4,
  \quad
 \lim_{|x|\to+\infty}
 |V(x)|=0,
 \quad
 -\frac{\nu}{|x|}-c_1\leq V\leq c_2:=\sup (V),
\end{equation}
with $\nu\in(0,1)$ and $c_1,c_2\geq 0$, $c_1+c_2-1<\sqrt{1-\nu^2}$.
The main key is  the following Hardy-type inequality:
\begin{equation}\label{eq:hardy.des}
\int_{\R^3} \frac{\abs{\sigma\cdot \nabla \varphi}^2}{1 + \frac{1}{\abs{x}}} + 
	\int_{\R^3} \left(1 - \frac{1}{\abs{x}}\right)\abs{\varphi}^2  \geq 0,
	\quad \text{ for all } \varphi \in C^\infty_c(\R^3;\C^2),
\end{equation}
see also 
\cite{dolbeault2004analytical} for a later direct analytical proof.
Using this inequality, Esteban and Loss considered general electrostatic potentials $\V$ in \cite{estebanloss} that satisfy \eqref{eq:des.potential} for $\nu \leq 1$, constructing the domain of self-adjointness for the operator $\Dirac_m + \V$.

More recently, \eqref{eq:hardy.des} has been generalized as follows \cite{els2020diraccoulomb1,els2020diraccoulomb2}
\[
\int_{\R^3} \frac{\left|\sigma\cdot \nabla \varphi\right|^2}{1+\mu*\frac{1}{|x|}} \geq \nu_0^2
	\int_{\R^3} \left(-1+ \mu*\frac{1}{|x|}\right)\abs{\varphi}^2
\]
for any probability measure $\mu$ and $\nu_0$ the optimal constant to determine that verifies $\frac{2}{\frac{\pi}{2}+\frac{2}{\pi}}\leq \nu_0\leq 1$. In detail, they state that proving $\nu_0 = 1$ would imply that the first eigenvalue inside the gap $(-m,m)$ of $\Dirac_m - \nu \mu * \frac{1}{|x|} \I_4$ is achieved only when $\mu=\delta_0$ is the Dirac delta centrer at the origin, see \cite{els23} for further details.

This paper aims to build on these results by proving a new weighted Hardy-type inequality for the Dirac operator, which admits minimizers and identifies the conditions under which the sharpest bound is attained.

\section{Main results}

In this section, we present the main result of this paper: a sharp, weighted Hardy-type inequality for the Dirac operator. This result incorporates the use of the \emph{spin angular momentum operator} $\S$ and the \emph{orbital angular momentum operator} $L$.
\begin{equation}\label{eq:defn.spin}
\mathbf{S}=
\frac{1}{2}\left(
\begin{array}{cc}
\sigma & 0\\
0 & \sigma
\end{array}
\right)\quad
\text{and}\quad
L:=-ix\wedge \nabla.
\end{equation}

\begin{theorem}\label{thm:hardy}
Let $m\geq 0$ and let $\omega,\eta \in L^1_{loc}$ be radial positive functions and assume that it exist $\tau\geq \frac{\gamma}2\geq 0$ such that
	\begin{align}
	\label{eq:condition-weight-sup}
 		\operatorname*{ess\,sup}_{x\in\Rt}&\
 		\left|\frac{\partial_r\omega\cdot\eta-\omega\cdot\partial_r\eta}{\omega}\right|:= \gamma \\
 		\label{eq:condition-weight-inf}
 		\operatorname*{ess\,inf}_{x\in\Rt} &\ \frac{\eta(x)}{|x|}:= \tau\quad
	\end{align} 
	where  $\partial_r$ denotes the radial derivative.
Then:
\begin{equation}\label{eq:hardy.con.phi}
\left(\tau-\frac{\gamma}{2}\right)^2 \int_{\Rt}\frac{|\psi|^2\omega}{\eta^2}\,dx
\leq
\left(\tau-\frac{\gamma}{2}\right)^2 \int_{\Rt}\frac{|(1+2\S\cdot L)\psi|^2\omega}{\eta^2}\,dx
\leq\int_{\Rt}\left|\Dirac_m\psi\right|^2\omega\,dx.
\end{equation}
The inequalities are sharp, in the sense that the constants on the left
hand side can not be improved.
\end{theorem}

\begin{remark}\label{thm:hardy.same}
Under the hypothesis of \Cref{thm:hardy}, assume further that $\omega=\eta$. In this case, \eqref{eq:condition-weight-sup} is verified with $\gamma=0$. Then, if $\eta=\omega$ verifies \eqref{eq:condition-weight-inf}, we obtain
\begin{equation}\label{eq:hardy.con.phi.same}
\tau^2\int_{\Rt}\frac{|\psi|^2}{\omega}\,dx
\leq
\tau\int_{\Rt}\frac{|(1+2\S\cdot L)\psi|^2}{\omega}\,dx
\leq\int_{\Rt}\left|\Dirac_m\psi\right|^2\omega\,dx.
\end{equation}
\end{remark}

\begin{theorem}
\label{thm:attainers-massless-exp}
Assume that $m = 0$. Define the weight functions $\omega(x) = \eta(x) = e^{|x|}$. With these choices, conditions \eqref{eq:condition-weight-sup} and \eqref{eq:condition-weight-inf} are satisfied with parameters $\gamma = 0$ and $\tau = e$, respectively.
Then the inequality  
\begin{equation}\label{eq:massless-exp}
e^2\int_{\Rt}\frac{|\psi|^2}{e^{|x|}}\,dx
\leq
e^2 \int_{\Rt}\frac{|(1+2\S\cdot L)\psi|^2}{e^{|x|}}\,dx
\leq\int_{\Rt}\left|-i\alpha\cdot\nabla\psi\right|^2e^{|x|},dx.
\end{equation}
admits as minimiser
\begin{equation}\label{eq:minimiser-massless-exp}
\psi_0(x)=
\begin{pmatrix}
c^-_{1/2,-1}\\
c^-_{-1/2,-1}\\
c^+_{1/2,-1}\\
c^+_{-1/2,-1}\\
\end{pmatrix}
\exp\left(e^{1-|x|}\right)
\quad
c^+_{\pm 1/2,\pm 1}\in\R.
\end{equation}
\end{theorem}

\begin{theorem}
\label{thm:attainers-massless-pol}
Assume that $m = 0$. Define the weight functions $\omega(x) = \eta(x) =(1+|x|)^a$, with $a>3$. With these choices, conditions \eqref{eq:condition-weight-sup} and \eqref{eq:condition-weight-inf} are satisfied with parameters $\gamma = 0$ and $\tau = \frac{a^a}{(a-1)^{a-1}}$, respectively.
Then the inequality  
\begin{equation}\label{eq:massless-pol}
\begin{split}
\left(\frac{a^a}{(a-1)^{a-1}}\right)^2\int_{\Rt}\frac{|\psi|^2}{(1+|x|)^a}\,dx
&\leq
\left(\frac{a^a}{(a-1)^{a-1}}\right)^2 \int_{\Rt}\frac{|(1+2\S\cdot L)\psi|^2}{(1+|x|)^a}\,dx\\
&
\leq\int_{\Rt}\left|-i\alpha\cdot\nabla\psi\right|^2(1+|x|)^a\,dx
\end{split}
\end{equation}
admits as minimiser
\begin{equation}\label{eq:minimiser-massless-pol}
\psi_0(x)=
\begin{pmatrix}
c^-_{1/2,-1}\\
c^-_{-1/2,-1}\\
c^+_{1/2,-1}\\
c^+_{-1/2,-1}\\
\end{pmatrix}
\exp\left[ \left(\frac{a}{a-1}\right)^a (1+|x|)^{1-a}\right]
\quad
c^+_{\pm 1/2,\pm 1}\in\R.
\end{equation}
\end{theorem}

\begin{remark}
As we can see in see \Cref{sec:minimisers}, we have a general method to construct weights $\omega,\eta$ for which inequality \eqref{eq:hardy.con.phi} admits non-trivial minimisers. Nevertheless we only chose some explicit examples in \Cref{thm:attainers-massless-exp} and \Cref{thm:attainers-massless-pol}, with exponential and polynomial decay respectively.
\end{remark}

\begin{remark}
We are only able to construct minimisers of \eqref{eq:hardy.con.phi} in the massless case when $m=0$, see \Cref{sec:minimisers} for details. This fact is somehow reminiscent of the huge difference between the massless ($m=0$) and the massive ($m>0$) Dirac operators, from a spectral theoretical point of view.
Indeed, in the latter we have a spectral gap which, together with the algebraic properties of the operator, is responsible of the appearance of eigenvalues whenever we add a Coulomb-type potential, see \cite[Section 8.7]{thaller2005advanced}. 
This phenomenon does not occur in the massless case where there is no spectral gap. This seems to suggest that, when $m=0$, there should be a suitable notion of \emph{small} perturbation at the critical scaling.
\end{remark}
\subsection*{Structure of the paper} 
This paper is organized as follows. 
In \Cref{sec:hardy}, we provide the proof of the weighted Hardy-Dirac inequality \eqref{eq:hardy.con.phi}
In \Cref{sec:minimisers} we construct the minimizers for this inequality. Finally, for completeness, in \Cref{sec:polar}, we review the \emph{partial wave decomposition} of the Dirac operator.

\section{Hardy-type inequalities for the Dirac Operator}
\label{sec:hardy}
To prove \eqref{eq:hardy.con.phi} we use the following abstract result.
\begin{lemma}\label{lem:abstract.lem}
  Let $\mathcal{S}, \mathcal{A}$ be respectively a symmetric and an anti-symmetric
operator on a complex Hilbert space $\mathcal{H}$. Then the following holds:
\[
  2\Re \langle
   \mathcal{A} u, 
  \mathcal{S} u
  \rangle_{\H}
  =
  \langle [\mathcal{S},\mathcal{A}] u,u\rangle_{\H},
\]
where $[\mathcal{S},\mathcal{A}]:=\mathcal{SA}-\mathcal{AS}$ is the 
commutator of the operators $\mathcal{S}$ and $\mathcal{A}$.
\end{lemma}
\begin{proof}
  The proof is a simple computation:
  \[
      2\Re \langle \mathcal{A} u,
       \mathcal{S} u
       \rangle_\H
       =
      \langle \mathcal{A} u,
       \mathcal{S} u
       \rangle_\H
       +
      \overline{\langle \mathcal{A} u,
       \mathcal{S} u
       \rangle_\H}
     = \langle  \mathcal{S}\mathcal{A} u,
       u
       \rangle_\H
       -
       \langle  \mathcal{A}\mathcal{S} u,
       u
       \rangle_\H.
      \qedhere
  \]
\end{proof}

\begin{proof}[Proof of \Cref{thm:hardy}]
Let us first the theorem assuming $\psi \in \mathcal S(\Rt,\C^4)$.
The general result will then follow by a standard density argument, see \cite{coulombluis} for details.
Denote $\hx:=x/|x|$ for any $x\neq 0$. Then the proof descends immediately from the explicit computation of the following:
\begin{equation}\label{eq:proof.hardy0}
\begin{split}
0&\leq \int_{\Rt}
	\bigg\lvert
		\Dirac_m\psi
        +\left(\tau-\frac{\gamma}{2}\right)(- i \alpha \cdot \hx)
        \frac{(1+ 2\SL)\psi}{\eta}
	\bigg\rvert^2 \omega\,dx
\\
&\leq
\int_{\Rt}
  \left|\Dirac_m \psi\right|^2\omega \,dx
-  \left(\tau-\frac{\gamma}{2}\right)^2 \int_{\R^3}
    \frac{|(1+2\SL)\psi|^2\omega}{\eta^2}\,dx.
\end{split}
\end{equation}
Indeed let $c>0$ be a constant to be determined, then
\begin{equation}\label{eq:proof.hardy1}
	\begin{split}
	0\leq & 
	\int_{\Rt} \bigg\lvert  (-i\alpha\cdot\nabla+m\beta) \psi+
	c(- i \alpha \cdot \hx)\cdot
    \frac{(1+ 2\SL)\psi}{\eta}\bigg\rvert^2\omega \,dx\\  
=&
  	\int_{\R^3}  | (-i\alpha\cdot\nabla+m\beta)\psi|^2\omega \,dx
  	+
  	c^2 \int_{\R^3}\frac{|(1+2\SL)\psi|^2\omega}{\eta^2}\,dx \\
  	& 
  	+2c\Re\int_{\R^3}
   (-i\alpha\cdot\nabla+m\beta)\psi
  \cdot
  \overline{(-i\alpha\cdot\hx)\frac{\omega}{\eta} (1+2\SL)\psi}
  \,dx.
\end{split}
\end{equation}
We state that the following sharp inequality holds
\begin{equation}\label{eq:doppioprod}
2\Re\int_{\R^3}
   (-i\alpha\cdot\nabla+m\beta)\psi
  \cdot
  \overline{(-i\alpha\cdot\hx)\frac{\omega}{\eta} (1+2\SL)\psi}
  \,dx
  \leq \left(\gamma-2\tau\right)\int_{\R^3}\frac{|(1+2\SL)\psi|^2\omega}{\eta^2} \,dx.
\end{equation}
This combined with \eqref{eq:proof.hardy1} gives:
\[
\int_{\Rt} \left| \Dirac_m \psi\right|^2\omega\, dx \geq 
	-\left(c^2+c(\gamma-2\tau)\right) \int_{\R^3}\frac{|(1+2\SL)\psi|^2\omega}{\eta^2} \,dx.
\]
Maximising the quantity $-(c^2+c(\gamma-2\tau))$, that is assuming $c=\tau-\frac{\gamma}{2}\geq 0$, we get \eqref{eq:proof.hardy0}.
This also proves that \eqref{eq:hardy.con.phi} is sharp.

To conclude the proof it remains to prove \eqref{eq:doppioprod}.
With an explicit computation (see for example \cite[Equation 4.102]{thaller}) we get that
\begin{equation}
\label{eq:alpha.grad.polar}
  -i\alpha\cdot \nabla = 
  - i \alpha \cdot \hx \left (\left(\partial_r + \frac{1}{\abs{x}}\right)\I_4
    - \frac{1}{|x|}(1+ 2\SL)\right),
  \end{equation}
 and so the right-hand side in \eqref{eq:doppioprod} can be expanded as follows:
\[
  \begin{split}
   2\Re\int_{\R^3}
  (-i\alpha&\cdot\nabla + m\beta)\psi
  \cdot
  \overline{(-i\alpha\cdot\hx)\frac{\omega}{\eta} (1+2\SL)\psi}
  \,dx   
  \\
      =\,&
   2\Re\int_{\R^3}
  \left(\partial_r + \frac{1}{\abs{x}}\right) \psi
  \cdot
  \overline{\frac{\omega}{\eta} (1+2\SL)\psi}
  \,dx 
  \\
  & 
  - 2\int_{\R^3} \frac{|(1+2\SL)\psi|^2\omega}{|x|\eta}\,dx
  \\
  &  +2m\Re\int_{\R^3}
  \beta \psi
  \cdot
  \overline{\frac{\omega}{\eta} (1+2\SL)\psi}
  \,dx
  \\
  =:& \,  I + II + III.
  \end{split}
\]
Let us analyse $I$. The operator $\left(\partial_r + \frac{1}{|x|}\right)$ is anti-symmetric. The operator
$\frac{\omega}{\eta} (1+2\SL)$ is symmetric, since the function $\frac{\omega}{\eta}$ is radial and it commutes with the symmetric operator $(1+2\SL)$.
Then, thanks to \Cref{lem:abstract.lem} and applying the Cauchy-Schwartz inequality:
\begin{equation}\label{eq:stima-I}
\begin{split}
I&=\int_{\Rt}
\left[\frac{\omega}{\eta} (1+2\SL),\partial_r + \frac{1}{|x|}\right]\psi\cdot\overline\psi\,dx=
	-\int_{\Rt} \partial_r\left(\frac{\omega}{\eta}\right)(1+2\SL)\psi\cdot\overline{\psi}\,dx\\
	&
	\leq
	\left(
	\int_{\Rt} \frac{|(1+2\SL)\psi|^2\omega}{\eta^2}\,dx
	\right)^{1/2}
	\cdot
	\left(
	\int_{\Rt} \left|\left(\frac{\partial_r\omega\cdot\eta-\omega\cdot\partial_r\eta}{\omega}\right)\psi\right|^2\frac{\omega}{\eta^2}\,dx
	\right)^{1/2}\\
	&\leq \gamma \int_{\Rt} \frac{|(1+2\SL)\psi|^2\omega}{\eta^2}\,dx,	
\end{split}	
\end{equation}
where, in the last inequality, we have used \eqref{eq:condition-weight-sup} and \eqref{eq:1<(1+SL)}. The sharpness of \eqref{eq:stima-I} derives from the fact that, by virtue of \eqref{eq:(1+2SL)=-kBeta}, equality holds when
\[
\psi(x)=f(|x|)\Phi^+_{\pm 1/2,1}(\hx)
\quad
\text{or}
\quad
\psi(x)=f(|x|)\Phi^-_{\pm 1/2,-1}(\hx)
\]
for $f(|x|)$ a radial function.

Let us estimate $II$. Using \eqref{eq:condition-weight-inf} we get
\[
II\leq   -2 \tau \int_{\R^3} \frac{|(1+2\SL)\psi|^2\omega}{\eta^2}\,dx
\]
Let us finally estimate $III$. Since the operators and $i\alpha\cdot\hx$ and $\beta$  anti-commute we rewrite
\[
    III = 2m\Re\int_{\R^3}
  \frac{\omega}{\eta}\psi
  \cdot
  \overline{i\alpha\cdot\hx \beta (1+2\SL)\psi}\,dx.
\]
Since the operator $\beta(1+2\SL)$ is symmetric, the operator $i\alpha\cdot\hx$ is anti-symmetric and they anti-commute (see \cite[Equation 4.108]{thaller}), we have that the operator $ i\alpha\cdot\hx  \beta(1+2\SL)$ is anti-symmetric.
Finally, the multiplication by $\frac{\omega}{\eta}$ is a symmetric operator, and it commutes with $i\alpha\cdot\hx  \beta(1+2\SL)$, being $\frac{\omega}{\eta}$ a radial function. Thus,  
\Cref{lem:abstract.lem} let us conclude that $III=0$, and so \eqref{eq:doppioprod} is proved.

\end{proof}

\section{Computation of the minimisers}
\label{sec:minimisers}

In this section, we adopt a slight abuse of notation: for any radial function $\phi$, we denote $\phi(r) = \phi(x)$, where $r = |x|$.

Let us assume that $\psi$ is a minimizer of \eqref{eq:hardy.con.phi}, that is

\begin{equation}\label{eq:hardy.con.phi.attainer}
\left(\tau-\frac{\gamma}{2}\right)^2 \int_{\Rt}\frac{|\psi|^2\omega}{\eta^2}\,dx=
\left(\tau-\frac{\gamma}{2}\right)^2 \int_{\Rt}\frac{|(1+2\S\cdot L)\psi|^2\omega}{\eta^2}\,dx
=\int_{\Rt}\left|\Dirac_m\psi\right|^2\omega\,dx.<+\infty.
\end{equation}

We can decompose $\psi$ as in \eqref{eq:dec.armonic}, that is
\[
\psi(x)=
  \sum_{j,\kappa_j,m_j} 
\frac{1}{r}\left(
f^+_{m_j,\kappa_j}(r)\Phi^+_{m_j,\kappa_j}(\hx)+
f^-_{m_j,\kappa_j}(r)\Phi^-_{m_j,\kappa_j}(\hx)\right).
\]

From the first equality of \eqref{eq:hardy.con.phi.attainer}, and thanks to \eqref{eq:(1+SL)radial}, we directly have $f^\pm_{m_j,k_j}=0$ for $\kappa_j\neq\pm 1$, or equivalently for $j\neq 1/2$. 

Let us focus on the second equality of \eqref{eq:hardy.con.phi.attainer}. Thanks to \eqref{eq:proof.hardy0} we get that $\psi$ solves the following equation  
\begin{equation}\label{eq:attainer-solve}
\left(-i\alpha\cdot\nabla+m\beta
        +\left(\tau-\frac{\gamma}{2}\right)(- i \alpha \cdot \hx)
        \frac{(1+ 2\SL)\psi}{\eta}\right)\psi=0
\end{equation}

Multiplying this equation by $i\alpha\cdot\hx$ and using \eqref{eq:alpha.grad.polar} we get
\begin{equation}\label{eq:quadrato.no.rad}
\left(
\partial_r+\frac{1}{|x|}
-\left(\frac{1}{|x|}-\frac{\tau-\gamma/2}{\eta}\right)(1+2\SL)
+i\alpha\cdot\hx(m\beta)
\right)
\psi=0
\end{equation}

The action of all the operators appearing in \eqref{eq:quadrato.no.rad} leaves invariant the decomposition in partial wave subspaces, see \Cref{sec:polar}. Thanks to \eqref{eq:(1+2SL)=-kBeta}, we get that for $m_{1/2}=\pm 1/2$ and $\kappa_{1/2}=\pm 1$ we have 
\begin{equation}\label{eq:quadrato.rad}
\begin{pmatrix}
\partial_r+\kappa_{1/2}\left(\dfrac{1}{r}-\dfrac{\tau-\gamma/2}{\eta(r)}\right)&-m\\
-m& \partial_r-\kappa_{1/2}\left(\dfrac{1}{r}-\dfrac{\tau-\gamma/2}{\eta(r)}\right)
\end{pmatrix}
\cdot
\begin{pmatrix}
f^+_{m_{1/2},\kappa_{1/2}}\\
f^-_{m_{1/2},\kappa_{1/2}}
\end{pmatrix}=0.
\end{equation}

We are now ready to prove \Cref{thm:attainers-massless-exp} and \Cref{thm:attainers-massless-pol}.

\begin{proof}[Proof of \Cref{thm:attainers-massless-exp}]
With these choices, we proceed to rewrite \eqref{eq:quadrato.rad}. For sake of clarity, we omit the subscript $1/2$, denoting $\kappa=\kappa_{1/2}\in \{-1,1\}$ and $m=m_{1/2}\in \{-1/2,1/2\}$. The equations reduce to the following system
\begin{equation}\label{eq:solve-exp}
\begin{pmatrix}
\partial_r+\kappa\left(\dfrac{1}{r}-\dfrac{e}{e^r}\right)&0\\
0& \partial_r-\kappa\left(\dfrac{1}{r}-\dfrac{e}{e^r}\right)
\end{pmatrix}
\cdot
\begin{pmatrix}
f^+_{m_{1/2},\kappa_{1/2}}\\
f^-_{m_{1/2},\kappa_{1/2}}
\end{pmatrix}=0.
\end{equation}
whose solutions are
\[
\begin{pmatrix}
f^+_{m,\kappa}\\
f^-_{m,\kappa}
\end{pmatrix}=
\begin{pmatrix}
c^+_{m,\kappa}\exp(-\kappa e^{1-r}) r^{-\kappa}
\\
c^-_{m,\kappa}\exp(\kappa e^{1-r}) r^{\kappa}
\end{pmatrix},
\quad\text{for}\ c^+_{m,\kappa},c^-_{m,\kappa}\in\R.
\]
To eliminate the dominant singularity at the origin, we set $c^+_{m,1}=c^-_{m,-1}=0$ and we get
\[
\begin{split}
\psi_0(x)=&\sum_{m=-1/2}^{1/2} 
\frac{1}{r}\left(
c^+_{m,-1} r \exp(e^{1-r}) \Phi_{m,-1}^+(\hx)
+
c^-_{m,1} r \exp(e^{1-r}) \Phi_{m,1}^-(\hx)
\right)
\\
&
=
\begin{pmatrix}
c^-_{1/2,-1}\\
c^-_{-1/2,-1}\\
c^+_{1/2,-1}\\
c^+_{-1/2,-1}\\
\end{pmatrix}
\exp(e^{1-|x|}).
\end{split}
\]
Identities \eqref{eq:solve-exp} and \eqref{eq:(1+SL)radial} ensure that 
\[
e^2\int_{\Rt}\frac{|\psi_0|^2}{e^{|x|}}\,dx
=
e^2 \int_{\Rt}\frac{|(1+2\S\cdot L)\psi_0|^2}{e^{|x|}}\,dx
=
\int_{\Rt}\left|-i\alpha\cdot\nabla\psi_0\right|^2e^{|x|}\,dx.
\]
Let us finally show that
\[
\int_{\Rt} |-i\alpha\cdot\nabla\psi_0|^2e^{|x|}\,dx<+\infty.
\]
Indeed, thanks to \eqref{eq:1<(1+SL)}
\[
\begin{split}
\int_{\Rt} |-i\alpha\cdot\nabla\psi|^2e^{|x|}\,dx
=&\sum_{m=-1/2}^{1/2} 
\left((c_{m,-1}^+)^2+
(c_{m,-1}^+)^2\right)\int_0^{+\infty} 
\left|\left(\partial_r-\frac{1}{r}\right)\left(r \exp(e^{1-r})\right)\right|^2 e^r\,dr\\
=&\sum_{m=-1/2}^{1/2} 
\left((c_{m,-1}^+)^2+(c_{m,-1}^+)^2\right)
\int_0^{+\infty} 
\exp(-r+2 e^{1-r}+2) r^2\,dr<+\infty,
\end{split}
\]
thus $\psi_0$ is an attainer of \eqref{eq:massless-exp}.

\end{proof}

\begin{proof}[Proof of \Cref{thm:attainers-massless-pol}]
With these choices, we proceed to rewrite \eqref{eq:quadrato.rad}. For sake of clarity, we omit the subscript $1/2$, denoting $\kappa=\kappa_{1/2}\in \{-1,1\}$ and $m=m_{1/2}\in \{-1/2,1/2\}$. The equations reduce to the following system
\[
\begin{pmatrix}
\partial_r+\kappa\left(\dfrac{1}{r}-\dfrac{\tau}{(1+r)^a}\right)&0\\
0& \partial_r-\kappa\left(\dfrac{1}{r}-\dfrac{\tau}{(1+r)^a}\right)
\end{pmatrix}
\cdot
\begin{pmatrix}
f^+_{m_{1/2},\kappa_{1/2}}\\
f^-_{m_{1/2},\kappa_{1/2}}
\end{pmatrix}=0,\quad
\text{for}\ \tau=\frac{a^a}{(a-1)^{a-1}}
\]
whose solutions are
\[
\begin{pmatrix}
f^+_{m,\kappa}\\
f^-_{m,\kappa}
\end{pmatrix}=
\begin{pmatrix}
c^+_{m,\kappa}\exp\left[-k \left(\frac{a}{a-1}\right)^a(1+r)^{1-a}\right] r^{-\kappa}
\\
c^-_{m,\kappa}\exp\left[k \left(\frac{a}{a-1}\right)^a(1+r)^{1-a}\right] r^{\kappa}
\end{pmatrix},
\quad\text{for}\ c^+_{m,\kappa},c^-_{m,\kappa}\in\R.
\]
Analogously as above, we set $c^+_{m,1}=c^-_{m,-1}=0$ and we get
\[
\begin{split}
\psi_0(x)=&\sum_{m=-1/2}^{1/2} 
\frac{1}{r}\Big(
c^+_{m,-1} r \exp\left[ \left(\frac{a}{a-1}\right)^a(1+r)^{1-a}\right] \Phi_{m,-1}^+(\hx)\\
&
\quad\quad\quad
+c^-_{m,1} r \exp\left[ \left(\frac{a}{a-1}\right)^a(1+r)^{1-a}\right] \Phi_{m,1}^-(\hx)
\Big)
\\
=&
\begin{pmatrix}
c^-_{1/2,-1}\\
c^-_{-1/2,-1}\\
c^+_{1/2,-1}\\
c^+_{-1/2,-1}\\
\end{pmatrix}
\exp\left[ \left(\frac{a}{a-1}\right)^a(1+|x|)^{1-a}\right]
\end{split}
\]
Identities \eqref{eq:solve-exp} and \eqref{eq:(1+SL)radial} ensure
the equality in \eqref{eq:minimiser-massless-pol}.
Let us finally show that
\[
\int_{\Rt} |-i\alpha\cdot\nabla\psi_0|^2(1+|x|)^a\,dx<+\infty.
\]
Indeed, thanks to \eqref{eq:1<(1+SL)}
\[
\begin{split}
&\int_{\Rt} |-i\alpha\cdot\nabla\psi_0|^2(1+|x|)^a\,dx
\\=&
\sum_{m=-1/2}^{1/2} 
\left((c_{m,-1}^+)^2+
(c_{m,-1}^+)^2\right)\int_0^{+\infty} 
\left|\left(\partial_r-\frac{1}{r}\right)\left(r \exp\left[ \left(\frac{a}{a-1}\right)^a(1+r)^{1-a}\right]\right)\right|^2 (1+r)^a\,dr\\
=&\sum_{m=-1/2}^{1/2} 
\left((c_{m,-1}^+)^2+(c_{m,-1}^+)^2\right)
\int_0^{+\infty} 
\left(\frac{a^a}{(a-1)^{a-1}}\right)^{2} \exp\left[ 2\left(\frac{a}{a-1}\right)^a(1+r)^{1-a}\right] \frac{r^2}{(r+1)^{a}}
\,dr\\
<&+\infty,
	\end{split}
\]
since $a>3$. Thus $\psi_0$ is a minimiser of \eqref{eq:massless-pol}.
\end{proof}

\appendix
\section{Partial wave subspaces}\label{sec:polar}
In this appendix, we recall the \emph{partial wave subspaces} associated to the Dirac equation. We sketch here this topic, referring to \cite[Section 4.6]{thaller} for further details.

Let $Y^l_n$ be the spherical harmonics. They are defined for $n = 0, 1, 2, \dots$, and $l =-n,-n + 1,\dots , n,$ and they satisfy $\Delta_{\mathbb{S}^2} Y^l_n= n(n + 1)Y^l_n$, where $\Delta_{\mathbb{S}^2}$ denotes the usual spherical Laplacian. Moreover, $Y^l_n$ form a complete orthonormal set in $L^2(\mathbb{S}^2)$.
For $j = 1/2, 3/2, 5/2, \dots , $ and $m_j = -j,-j + 1, \dots , j$, set
\[
\Psi^{m_j}_{j-1/2}:=
\frac{1}{\sqrt{2j}}
\left(\begin{array}{c}
\sqrt{j+m_j}\,Y^{m_j-1/2}_{j-1/2}\\
\sqrt{j-m_j}\,Y^{m_j+1/2}_{j-1/2}\\
\end{array}\right),
\quad
\Psi^{m_j}_{j+1/2}:=\frac{1}{\sqrt{2j+2}}
\left(\begin{array}{c}
\sqrt{j+1-m_j}\,Y^{m_j-1/2}_{j+1/2}\\
-\sqrt{j+1+m_j}\,Y^{m_j+1/2}_{j+1/2}\\
\end{array}\right);
\]
then  $\psi^{m_j}_{j\pm1/2}$ form a complete orthonormal set in $L^2(\mathbb{S}^2)^2$.
For $\kappa_j:=\pm(j+1/2)$ we set
\[
\Phi^+_{m_j,\pm(j+1/2)}:=
\left(\begin{array}{c}
i\,\Psi^{m_j}_{j\pm1/2}\\
0
\end{array}\right),
\quad
\Phi^-_{m_j,\pm(j+1/2)}:=
\left(\begin{array}{c}
0\\
\Psi^{m_j}_{j\mp1/2}
\end{array}\right).
\]
Then, the set $\seq{\Phi^+_{m_j,\kappa_j},\Phi^-_{m_j,\kappa_j}}_{j,\kappa_j,m_j}$ is a 
complete orthonormal basis of $L^2(\mathbb{S}^2;\C^4)$ and it is left invariant by the operators $(1+2\SL)$, $\beta$ and $-i\alpha\cdot\hx$. With respect to this basis these operators are represented by the $2\times 2$ matrices:
\begin{gather}
\label{eq:(1+2SL)=-kBeta}
  (1+2\SL)=
  \begin{pmatrix}
  -\kappa_j &0\\
  0& \kappa_j
  \end{pmatrix},\quad
  \beta =  
  \begin{pmatrix}
  1 &0\\
  0&-1
  \end{pmatrix}
  \quad 
  i\alpha\cdot\hx=
    \begin{pmatrix}
  0&1\\
  -1&0
  \end{pmatrix}
\end{gather}
where the \emph{spin angular momentum operator} $\S$ and the \emph{orbital angular momentum} $L$ are defined in \eqref{eq:defn.spin}

\verde{We define the following space:
\begin{equation}
\mathcal{H}_{m_j,\kappa_j}:=
\seq*{ 
\frac{1}{r}\left(
f^+_{m_j,\kappa_j}(r)\Phi^+_{m_j,\kappa_j}(\hx)+
f^-_{m_j,\kappa_j}(r)\Phi^-_{m_j,\kappa_j}(\hx)\right)
\in L^2(\Rt)
\mid
f^\pm_{m_j,\kappa_j}\in L^2(0,+\infty)}.
\end{equation}}
So, we can write
\begin{equation}\label{eq:dec.armonic}
\psi(x)=
  \sum_{j,\kappa_j,m_j} 
\frac{1}{r}\left(
f^+_{m_j,\kappa_j}(r)\Phi^+_{m_j,\kappa_j}(\hx)+
f^-_{m_j,\kappa_j}(r)\Phi^-_{m_j,\kappa_j}(\hx)\right)
\end{equation}
and, by definition,
\[
\int_{\Rt}|\psi|^2\,dx=  
\sum_{j,\kappa_j,m_j} 
\int_0^{+\infty}|f^+_{m_j,\kappa_j}(r)|^2+
|f^-_{m_j,\kappa_j}(r)|^2\,dr.
\]
Thanks to \cite[Equations 4.109 and 4.129]{thaller} and \eqref{eq:dec.armonic}, we have that for any radial weight $\omega$
\begin{equation}\label{eq:(1+SL)radial}
\begin{gathered}
\int_{\Rt}
|-i\alpha\cdot\nabla\psi|^2\omega\,dx
=
  \sum_{j,\kappa_j,m_j} 
\int_0^{+\infty}
\left(
\left|\left(\partial_r+\frac{\kappa_j}{r}\right) f^+_{m_j,\kappa_j}(r)\right|^2+
\left|\left(\partial_r-\frac{\kappa_j}{r}\right) f^-_{m_j,\kappa_j}(r)\right|^2
\right)\omega(r)\,dr,\\
\int_{\Rt}
|\psi|^2\omega\,dx
=
  \sum_{j,\kappa_j,m_j} 
\int_0^{+\infty}
\left(
|f^+_{m_j,\kappa_j}(r)|^2+
|f^-_{m_j,\kappa_j}(r)|^2\right)\omega(r)\,dr,
\\
\int_{\Rt}
|(1+2\SL)\psi|^2\omega\,dx
=
  \sum_{j,\kappa_j,m_j} 
\int_0^{+\infty}
\kappa_j^2
\left(
|f^+_{m_j,\kappa_j}(r)|^2+
|f^-_{m_j,\kappa_j}(r)|^2\right)\omega(r)dr.
\end{gathered}
\end{equation}
From \eqref{eq:(1+SL)radial}, we directly deduce that
\begin{equation}\label{eq:1<(1+SL)}
\int_{\Rt}
\frac{|\psi|^2}{|x|}\,dx
\leq
\int_{\Rt}
\frac{|(1+2\SL)\psi|^2}{|x|}\,dx,
\end{equation}
and that \eqref{eq:1<(1+SL)} is attained if and only if
$f_{m_j,\kappa_j}^\pm=0$  for $\kappa_j\neq\pm 1$, or equivalently $j\neq 1/2$.

\section*{Acknowledgements}
We would like to thank Luis Vega for the enlightening discussions.

\section*{Funding}
L.~Fanelli and F.~Pizzichillo are partially supported by the Basque Government through the BERC 2022-2025 program, by the Ministry of Science and Innovation: BCAM Severo Ochoa accreditation CEX2021-001142-S / MICIN / AEI / 10.13039/501100011033  PID2021-123034NB-I00 funded by funded by MCIN/ AEI/10.13039/501100011033/ FEDER, UE.
L.~Fanelli is also supported by the projects IT1615-22 funded by the Basque Government, and by Ikerbasque.

\section*{Data availability statement}
No new data were created or analysed in this study.

\section*{Conflict of interest}
The authors have no relevant financial or non-financial interests
to disclose.

\end{document}